\documentclass[review]{elsarticle}
\usepackage{lineno,hyperref,mathrsfs}
\usepackage{amsmath,amssymb,amsthm,color,esint}
\usepackage{enumitem}
\usepackage{CJK}
\newtheorem{definition}{Definition}[section]
\newtheorem{theorem}{Theorem}[section]

\newtheorem{proposition}{Proposition}[section]

\newtheorem{remark}{Remark}[section]

\def\punct{}
\newtheoremstyle{dotless}{}{}{\rm}{}{\bf}{\punct}{.5em}{}
\theoremstyle{dotless}

\def\R {\mathbb{R}}
\def\d{{\,\rm d}}

\def \leq{\leqslant}
\def \geq{\geqslant}
\numberwithin{equation}{section}
\journal{arXiv}

\begin{document}
	
\begin{frontmatter}
	
	\title{Small-time approximate controllability for the nonlinear complex Ginzburg-Landau equation with bilinear control}

	
	\medskip
	
	\author[my3address]{Xingwu Zeng}
	\ead{xingwuzeng@whu.edu.cn}
	
	\author[my3address]{Can Zhang}
\ead{canzhang@whu.edu.cn}

\address[my3address]{School of Mathematics and Statistics, Wuhan University, Wuhan 430072, China.}

	\begin{abstract}
		In this paper, we consider the bilinear approximate controllability for  the complex Ginzburg-Landau (CGL) equation with a power-type nonlinearity of any integer degree on a torus of  arbitrary space dimension. Under a saturation hypothesis on the control operator, we show  the small-time global controllability of the CGL equation. The proof is obtained by developing a multiplicative version of a geometric control approach, introduced by Agrachev and Sarychev in \cite{AS05,AS06}.
	\end{abstract}
	
\end{frontmatter}



\section{Introduction}
\subsection{Model and controllability concepts}
The complex Ginzburg-Landau (CGL) equation is classical in the theory of amplitude equations, and is a simple but important model of convection, flow, and turbulence, etc. It has been studied by many mathematicians and physicists, see \cite{GL,NewellWhitehead,Segel,SteStu}.
We are interested in the following CGL equation  on the torus $\mathbb{T}^d=\mathbb{R}^d / 2 \pi \mathbb{Z}^d$:
\begin{equation}\label{(00)}
\partial_t \psi=V \psi+(1+i \nu) \Delta \psi-(1+i \mu)|\psi|^{2 \sigma} \psi,
\end{equation}where $d,\, \sigma \geq 1$ are arbitrary integers, $i=\sqrt{-1}$, $V\geq0$ and  $\nu,\, \mu\in \R$.
Without loss of generality, \eqref{(00)} has been normalized so that the coefficients of the linear and nonlinear dissipation (damping) terms are unity. $V$ is the coefficient of the linear driving term. $\sigma$ sets the degree of the nonlinearity. $\nu$ and $\mu$ are the coefficients of the linear and nonlinear dispersive terms respectively.

Let $ q \geq 1$ be arbitrary integer, and  $Q: \mathbb{T}^d \rightarrow \mathbb{R}^q$ be a given smooth external field. We consider the following CGL equation with bilinear control on the torus $\mathbb{T}^d$:
\begin{equation} \label{(0.0)}
	\partial_t \psi=V \psi+(1+i \nu) \Delta \psi-(1+i \mu)|\psi|^{2 \sigma} \psi +(1+i R) \langle u(t), Q(x)\rangle \psi .
\end{equation}
The pair $ (u,R)\in L^2_{\text{loc}}(\R^+;\R^{q})\times\R$ plays the role of a control. 
\eqref{(0.0)} is equipped with the initial condition
\begin{equation}\label{(0.2)}
\psi(0, x)=\psi_0(x)
\end{equation}
belonging to a Sobolev space $H^s:=H^s\left(\mathbb{T}^d ; \mathbb{C}\right)$ of order $s>d / 2$, so that the problem is locally well-posed in $H^s$ (see Proposition~\ref{prop1.1} below). 

We next introduce the controllability properties that we are interested in. The first one is the small-time global approximate null-controllability.
\begin{definition}\label{def:control0}
The system \eqref{(0.0)}, \eqref{(0.2)} is small-time approximately null-controllable, if for any $ \varepsilon>0,\, T>0,\,\psi_0 \in H^s$, there exists a control $(u,R) \in L^2([0,T];\R^q)\times\R$ such that
\begin{equation*}
	\left\|\psi(T)\right\|_{H^s}<\varepsilon.
\end{equation*}
\end{definition}
The second one is the small-time global approximate null-controllability of phases, i.e., the small-time global approximate null-controllability to any $H^s$-neighbourhood of any target  of the form $ e^{(1-i\nu) \theta} \psi_0$.
\begin{definition}\label{def:control1}
The system \eqref{(0.0)}, \eqref{(0.2)} is approximately controllable for any $\varepsilon>0,\, \mathcal{T}>0,\, \psi_0 \in H^s$, and $\theta \in C^{\infty}\left(\mathbb{T}^d ; \mathbb{R}\right)$, there is a time $T \in(0, \mathcal{T})$, a control $(u,R) \in C^{\infty}\left([0, T] ; \mathbb{R}^{q}\right)\times\R$, and a unique solution $\psi \in C\left([0, T] ; H^s\right)$ of \eqref{(0.0)}, \eqref{(0.2)} such that
\begin{equation*}
\left\|\psi(T)-e^{(1-i\nu) \theta} \psi_0\right\|_{H^s}<\varepsilon.
\end{equation*}
\end{definition}

\subsection{Notation}
In this paper, we will use the following notation.
\begin{itemize}
	\item $\langle\cdot, \cdot\rangle$ is the Euclidian scalar product in $\mathbb{R}^q$ and $\|\cdot\|$ is the corresponding norm.
	\item $H^s=H^s\left(\mathbb{T}^d ; \mathbb{C}\right)\text{ with }s \geq 0$ and $L^p=L^p\left(\mathbb{T}^d ; \mathbb{C}\right)\text{ with }p \geq 1$, are the standard Sobolev and Lebesgue spaces of functions $f: \mathbb{T}^d \rightarrow \mathbb{C}$ endowed with the norms $\|\cdot\|_s$ and $\|\cdot\|_{L^p}$. 
	The space $ H^s $ is endowed with the norm
	\begin{equation*}
		\|f\|_{s}^2=\sum_{k \in \mathbb{Z}^d}\left(1+|k|^2\right)^s|\widehat{f}(k)|^2,
	\end{equation*}where
	\begin{equation*}
		|k|^2=k_1^2+k_2^2+\ldots+k_d^2,\quad \forall k=\left(k_1,\ldots,k_d\right)\in \mathbb{Z}^d
	\end{equation*}and the Fourier coefficient
	\begin{equation*}
		\widehat{f}(k)=\int_{\mathbb{T}^d} f(x) e^{-i \langle k , x\rangle} \d \mathfrak{m}(x)=\frac{1}{(2 \pi)^d} \int_{[0,2\pi]^d} f(x) e^{-i \langle k , x\rangle} \d x.
	\end{equation*}
	The space $L^2$ is endowed with the scalar product
	$$
	\langle f, g\rangle_{L^2}=\int_{\mathbb{T}^d} f(x) \overline{g(x)} \d \mathfrak{m}(x) =\frac{1}{(2 \pi)^d}\int_{[0,2\pi]^d} f(x) \overline{g(x)} \d x=\frac{1}{(2 \pi)^d} \langle f, g\rangle_{L^2([0,2\pi]^d)} .
	$$
	\item $C^s=C^s\left(\mathbb{T}^d ; \mathbb{C}\right)$ with $s \in \mathbb{N} \cup\{\infty\}$, is the space of $s$-order continuously differentiable function.
	\item Let $X$ be a Banach space. We denote by $B_X(a, r)$ the closed ball of radius $r>0$ centred at $a \in X$.
	\item  We write $J_T$ instead of $[0, T]$ and $J$ instead of $[0,1]$.
	\item $C\left(J_T ; X\right)$ is the space of continuous functions $f: J_T \rightarrow X$ with the norm
	$$
	\|f\|_{C\left(J_T ; X\right)}=\max _{t \in J_T}\|f(t)\|_X .
	$$
	\item $L^p\left(J_T ; X\right)\text{ with }1 \leq p<\infty$, is the space of Borel-measurable functions $f: J_T \rightarrow X$ with
	$$
	\|f\|_{L^p\left(J_T ; X\right)}=\left(\int_0^T\|f(t)\|_X^p \d t\right)^{1 / p}<\infty.
	$$
	\item $s_d$ is the smallest integer strictly greater than $d / 2$.
	\item $\mathbf{1}$ is the function identically equal to 1 on $\mathbb{T}^d$.
\end{itemize}

\subsection{Main results}
The purpose of this paper is to study the CGL equation \eqref{(0.0)} when the driving force $(u,R)$ acts multiplicatively through only few low Fourier modes. 
Let $K \subset \mathbb{Z}_*^d:= \mathbb{Z}^d\backslash\left\{(0, \ldots, 0)\right\}$ be the set of $d$ vectors defined by
\begin{equation}\label{(0.3)}
K=\{(1,0, \ldots, 0),\,(0,1, \ldots, 0),\, \ldots,\,(0,0, \ldots, 1,0),\,(1, \ldots, 1)\}.
\end{equation}
Assume that the field $Q=\left(Q_1, \ldots, Q_q\right)$ satisfies
\begin{equation}\label{(0.4)}
\{\mathbf{1},\, \sin \langle x, k\rangle,\, \cos \langle x, k\rangle\mid k \in K\} \subset \operatorname{span}\left\{Q_j\mid j=1, \ldots, q\right\}.
\end{equation}

The main results of this paper are as following.
\begin{theorem}\label{thm0}
	Assume that the condition \eqref{(0.4)} is satisfied. Let $s \geq s_d$ be an integer. The system \eqref{(0.0)}, \eqref{(0.2)} is small-time approximately null-controllable for any $R\in\R$ in the sense of Definition~\ref{def:control0}. 
\end{theorem}
\begin{theorem}\label{thma}
	Assume that the condition \eqref{(0.4)} is satisfied. Let $s \geq s_d$ be an integer. The system \eqref{(0.0)}, \eqref{(0.2)} is approximately controllable in the sense of Definition~\ref{def:control1}.
\end{theorem}
Several remarks are given in order.
\begin{remark}
	We emphasize that the approximate null-controllability in Theorem~\ref{thm0} does not impose any condition on the scalar control $R$,  i.e., the choice of $R\in\R$ is arbitrary.
\end{remark}
\begin{remark}
A more general formulation of the result in Theorem~\ref{thma} is given in Theorem~\ref{thm2.2}, where the controllability is proved under an abstract saturation condition for the field $Q$ (see the condition~{\hyperlink{H1}{\textbf{($\mathscr{P}$)}}}).
We will see from the proof of Theorem~\ref{thm2.2} that the scalar control $  R=-\nu $.
Note that the time $T$ may depend on the initial condition $\psi_0$, the target $e^{(1-i\nu) \theta} \psi_0$, and the parameters in this equation. 
\end{remark}
\begin{remark}
	As the bilinear approximate null-controllability has been derived in Theorem~\ref{thm0}, one may naturally ask whether the bilinear exact null-controllability can be achieved. Unfortunately, the bilinear exact null-controllability seems to be impossible. For the sake of simplicity, let us consider the solution $\psi$ of the following heat equation with the bilinear control $u\in L^\infty(\R^d\times(0,T))$:
	\begin{equation*}
		\left\{\begin{aligned}
			\partial_t \psi-\Delta \psi-u \psi&=0 \quad &&\text { in } \R^d \times(0, T), \\
			\psi\mid_{t=0}&=\psi_0 \quad &&\text { in } \R^d,
		\end{aligned}\right.
	\end{equation*}where the initial datum $\psi_0\not\equiv0$. Then, if $\psi(x,T)\equiv0$ in $\R^d$ for some $u\in L^\infty(\R^d\times(0,T))$ and $T>0$, the well-known backward uniqueness for parabolic equations (see \cite{LuisCPAM2003}) implies that $\psi$ vanishes identically in $\R^d\times[0,T]$. This contradicts $\psi_0\not\equiv0$. The conclusion is that, if $\psi_0\not\equiv0$, then for any $u\in L^\infty(\R^d\times(0,T))$ and $T>0$, $\psi(x,T)\not\equiv0$ in $\R^d$.
\end{remark}
\begin{remark}
	 The small-time controllability of phases is proved in Theorem~\ref{thma}, while a stronger controllability result can be considered: the small-time $L^2$-approximate controllability. It refers to the possibility, for any $\varepsilon>0$ and every $\psi_0,\, \psi_1\in L^2$ with $\|\psi_0\|_{L^2}=\|\psi_1\|_{L^2}=1$, there exists a time $T\in[0,\varepsilon]$, a global phase $\theta\in[0,2\pi)$ and controls $u: [0, T] \rightarrow \mathbb{R}^{q},\,R\in\R$ such that the unique solution $\psi \in C\left([0, T] ; H^s\right)$ of \eqref{(0.0)}, \eqref{(0.2)} satisfies\begin{equation*}
		\|	\psi(T;\psi_0,u,R)-e^{i\theta}\psi_1\|_{L^2}<\varepsilon.
	\end{equation*}
	We refer the interested readers to \cite{KB_EP_2}.
\end{remark}

\subsection{Methodology}
The main idea is motivated by \cite{Duca_EMS}. To draw a picture of the proofs of Theorems~\ref{thm0} and \ref{thma}, let us first turn to the following Cauchy problem: 
\begin{equation}\label{(0.1)}
	\partial_t \psi=V \psi+(1+i \nu) \Delta \psi-(1+i \mu)|\psi|^{2 \sigma} \psi +(r_1+i r_2) \langle u(t), Q(x)\rangle \psi ,
\end{equation}
which is equipped with the same initial condition \eqref{(0.2)}.
We denote by $\mathcal{R}_t\left(\psi_0, (u,r_1,r_2)\right)$ solution of the system \eqref{(0.1)}, \eqref{(0.2)} defined up to some maximal time. We denote by $\mathbb{B}(\varphi)(x)=\sum_{j=1}^d\left(\partial_{x_j} \varphi(x)\right)^2$. 
A central role in the proof is played by the limit
\begin{equation}\label{(0.5)}
	e^{ (a+ib)\delta^{-1 / 2} \varphi} \mathcal{R}_\delta\left(e^{-(a+ib) \delta^{-1 / 2} \varphi} \psi_0, (\delta^{-1} u,r_1,r_2)\right) \rightarrow e^{(r_1+i r_2)\left( \mathbb{B}(\varphi)+\langle u, Q\rangle\right)} \psi_0  \quad \text { in } H^s \text { as } \delta \rightarrow 0^{+},
\end{equation}
which holds for any $\psi_0 \in H^s$, non-negative function $\varphi \in C^{\infty}\left(\mathbb{T}^d ; \mathbb{R}\right)$,  constants $(u,r_1,r_2) \in \mathbb{R}^q\times\R\times\R$, and $ a,b\in\R $ satisfying $ r_1+ir_2=(1+i\nu)(a+ib)^2. $ The limit \eqref{(0.5)} specifies the asymptotic behavior of the solution of the CGL equation in small time under appropriately scaled large control and rapidly oscillating initial condition. Theorem~\ref{thm0} is derived directly from \eqref{(0.5)}. 

In order to explain the main idea, we consider the 1-D case, $d=1$. In this case, we could choose the field $Q=\left(\mathbf{1},\,\sin x,\,\cos x\right)$ such that assumption \eqref{(0.4)} satisfied, and \eqref{(0.5)} becomes 
\begin{equation}\label{(0.5)1d}
	e^{ (a+ib)\delta^{-1 / 2} \varphi} \mathcal{R}_\delta\left(e^{-(a+ib) \delta^{-1 / 2} \varphi} \psi_0, (\delta^{-1} u,r_1,r_2)\right) \rightarrow e^{(r_1+i r_2)\left( \varphi_x^2+\langle u, Q\rangle\right)} \psi_0  \quad \text { in } H^s \text { as } \delta \rightarrow 0^{+}.
\end{equation}
Letting $r_1=1,r_2=-\nu$, applying this limit with $\varphi=0$ and using the assumption \eqref{(0.4)}, we see that \eqref{(0.1)} can be controlled in small time from any initial point $\psi_0\in H^s$ to an arbitrary neighbourhood of $e^{(1-i\nu) \theta} \psi_0$ for any $\theta$ in the vector space
$$
H_0=\operatorname{span}\{\mathbf{1},\, \sin  x,\, \cos  x\}.
$$
Applying again the limit \eqref{(0.5)1d} with non-negative functions $\varphi=\theta_j \in H_0,\, j=0, \ldots, n$, we add more directions in $\theta$. Namely, the system can be steered from $\psi_0$ close to $e^{(r_1+ir_2) \theta} \psi_0$, where $\theta$ now belongs to a larger vector space $H_1$ whose elements are of the form
$$
\theta_0+\sum_{j=1}^n (\theta_j^\prime)^2.
$$
We iterate this argument and construct an increasing sequence $\left\{H_j\right\}$ of subspaces such that the equation can be approximately controlled to any target $e^{(r_1+ir_2) \theta} \psi_0$ with any $\theta \in H_j$ and $j \geq 1$. Since the union $\bigcup_{j=1}^{\infty} H_j$ is dense in $C^k\left(\mathbb{T}^d ; \mathbb{R}\right)$ for any $k \geq 1$, this concludes the proof of approximate controllability of problem~\eqref{(0.1)} and \eqref{(0.2)} with $r_1=1,\, r_2=-\nu$ and $u\in C^{\infty}\left([0, T] ; \mathbb{R}^{q}\right)$, thereby establishing Theorem~\ref{thma}.

\subsection{Related literature}
The controllability of the CGL equation has been widely studied via additive controls, see \cite{DFLZ_23,Fu_07,FuLiu_17,RZ_09}. Most of the controllability results are proved by establishing Carleman estimates.
However, to the best of our knowledge, the controllability of the CGL equation via bilinear control has not been discussed.

The present paper is the first one to deal with the bilinear controllability for the CGL equation by using the Agrachev-Sarychev type approach. Agrachev and Sarychev \cite{AS05,AS06,AS08} first developed the saturating geometric control approach when studying the approximate controllability of the 2D Navie-Stokes  and Euler systems by finite-dimensional forces. Later, their approach has been further extended to different equations in the case of additive controls, see \cite{Nersesyan21,CPAM2025,JMPA2024,Sarychev12,Shirikyan06,Shirikyan18}. Recently, this approach has been implemented for bilinear small-time  controllability of the Schr\"odinger equation in \cite{KB_EP_2,Duca_EMS},  the heat equation in \cite{Duca_JMPA}, and the Burgers equation in \cite{Duca_Burgers}. The intersted readers can find additional results on small-time controllability of PDEs through bilinear control by similar methods in \cite{CP_Auto,CoronXZ_JDE,DucaNersesyan_SICON,Duca_SICON,Pozzoli_JDE}.

\subsection{Organization}
The rest of this paper is organised as follows. In Section~\ref{sec:Preliminaries}, we establish the local well-posedness, some stability properties and a key asymptotic property  of the CGL equation. In Section~\ref{sec:Approximate controllability}, we formulate more general versions of Theorems~\ref{thm0} and \ref{thma}  and give their proofs. At the end of Section~\ref{sec:Approximate controllability}, we give an example of a saturating subspace and prove Theorems~\ref{thm0} and \ref{thma}. Section~\ref{sec:conclud} gives some concluding comments. Appendix is devoted to the proof of Proposition~\ref{prop1.1}.

\subsection{Acknowledgements}
This work was partially supported  by the National Natural Science Foundation of China (NSFC) grant 12422118.

\section{Two auxiliary propositions}\label{sec:Preliminaries}
\subsection{Local well-posedness}
In the following, we consider the CGL equation \eqref{(0.1)}, where $u$ is a  $\mathbb{R}^q$-valued function and $Q: \mathbb{T}^d \rightarrow \mathbb{R}^q$ are arbitrary smooth functions. 
We shall always assume that the parameters $d,\, \sigma\geq 1$, and $V\geq0,\, \nu,\, \mu,\, r_1,\, r_2\in \R$ are arbitrary. We next present two propositions that will be used in the proofs of our main results. 
The first proposition is about the local well-posedness and stability of the CGL equation in suitable Sobolev spaces. The proof is standard, hence we give it in Appenix.
\begin{proposition}\label{prop1.1}
For any $s>d / 2,\, \tilde{\psi}_0 \in H^s$, and $\tilde{u} \in L_{\text{loc}}^2\left(\mathbb{R}_{+} ; \mathbb{R}^q\right)$, there is a maximal time $\tilde{T}=\tilde{T}\left(\tilde{\psi}_0, \tilde{u}\right)>0$ and a unique solution $\tilde{\psi}$ of the problem \eqref{(0.1)}, \eqref{(0.2)} with $\left(\psi_0, u\right)=\left(\tilde{\psi}_0, \tilde{u}\right)$, whose restriction to the interval $J_T$ belongs to $C\left(J_T ; H^s\right)$ for any $T<\tilde{T}$. If $\tilde{T}<\infty$, then $\|\tilde{\psi}(t)\|_s \rightarrow \infty$ as $t \rightarrow \tilde{T}^{-}$. Furthermore, for any $T<\tilde{T}$, there are constants $\delta=\delta(T, \Lambda)>0$ and $C=C(T, \Lambda)>0$, where
$$
\Lambda=\|\tilde{\psi}\|_{C\left(J_T ; H^s\right)}+\|\tilde{u}\|_{L^2\left(J_T ; \mathbb{R}^q\right)},
$$
such that the following properties hold:\\
(i) For any $\psi_0 \in H^s$ and $u \in L^2\left(J_T ; \mathbb{R}^q\right)$ satisfying
\begin{equation}\label{stability}
 \|\psi_0-\tilde{\psi}_0 \|_s+\|u-\tilde{u}\|_{L^2\left(J_T ; \mathbb{R}^q\right)}<\delta,
\end{equation}
the problem \eqref{(0.1)}, \eqref{(0.2)} has a unique solution $\psi \in C\left(J_T ; H^s\right)$.\\
(ii) Let $\mathcal{R}$ be the resolving operator for \eqref{(0.1)}, i.e., the mapping taking a couple $\left(\psi_0, u\right)$ satisfying \eqref{stability} to the solution $\psi$. Then there exists a constant $C>0$ such that
\begin{equation}\label{continuity}
\left\|\mathcal{R}\left(\psi_0, u\right)-\mathcal{R}\left(\tilde{\psi}_0, \tilde{u}\right)\right\|_{C\left(J_T ; H^s\right)} \leq C\left( \|\psi_0-\tilde{\psi}_0 \|_s+\|u-\tilde{u}\|_{L^2\left(J_T ; \mathbb{R}^q\right)}\right) .
\end{equation}
\end{proposition}
\begin{remark}
	The global well-posedness of problem \eqref{(0.1)}, \eqref{(0.2)} is rather delicate even without any control.  More precisely, the global well-posedness of \eqref{(00)} can be established in the one-dimensional case, while in higher dimensions, the global well-posedness of \eqref{(00)} depends on specific choices of $\mu$ and $\nu$, we refer the interested readers to \cite[Theorem~3.1.39]{GuoJiangLi}.
\end{remark}

\subsection{Small-time asymptotic property}
Before formulating the second proposition, let us introduce some notations. For any $\psi_0 \in H^s$ and $T>0$, recall $\Theta\left(\psi_0, T\right)$ be the set of functions $u \in L^2\left(J_T ; \mathbb{R}^q\right)$ such that the system \eqref{(0.1)}, \eqref{(0.2)} has a solution in $C\left(J_T ; H^s\right)$. By Proposition~\ref{prop1.1}, the set $\Theta\left(\psi_0, T\right)$ is open in $L^2\left(J_T ; \mathbb{R}^q\right)$. For any $\varphi \in C^1\left(\mathbb{T}^d ; \mathbb{R}\right)$, recall that
\begin{equation}\label{(1.2)}
\mathbb{B}(\varphi)(x)=\sum_{j=1}^d\left(\partial_{x_j} \varphi(x)\right)^2 .
\end{equation}
For convenience, we choose $ a,\, b\in\R $ such that $$ r_1+ir_2=(1+i\nu)(a+ib)^2. $$
The following asymptotic property plays a key role in this paper.
\begin{proposition}\label{prop1.2}
For any integer $s \geq s_d,\, \psi_0 \in H^s,\, (u,r_1,r_2) \in \mathbb{R}^q\times\R\times\R$, and  non-negative function $\varphi \in C^r\left(\mathbb{T}^d ; \mathbb{R}\right)$ with $r=s+2$, there is a constant $\delta_0>0$ such that, for any $\delta \in\left(0, \delta_0\right)$, we have\footnote{For any vector $u \in \mathbb{R}^q$, with a slight abuse of notation, we denote by the same letter the constant value function equals to $u$.} $\delta^{-1} u \in \Theta\left(e^{-(a+ib) \delta^{-1 / 2} \varphi} \psi_0, \delta\right)$ and
\begin{equation*}\label{(1.3)}
e^{ (a+ib)\delta^{-1 / 2} \varphi} \mathcal{R}_\delta\left(e^{-(a+ib) \delta^{-1 / 2} \varphi} \psi_0, (\delta^{-1} u,r_1,r_2)\right) \rightarrow e^{(r_1+i r_2)\left( \mathbb{B}(\varphi)+\langle u, Q\rangle\right)} \psi_0 \quad \text { in } H^s \text { as } \delta \rightarrow 0^{+}.
\end{equation*}
Here $\mathcal{R}_\delta$ is the restriction of the solution at time $t=\delta$.
\end{proposition}
\begin{remark}
This result is inspired by \cite[Proposition~1.2]{Duca_EMS} and \cite[Proposition~2.3]{Duca_JMPA} (see also \cite[Proposition~2]{Nersesyan21}). However, compared to the results in these two papers, one may observe that both the initial state and the target state oscillate, and their moduli scale simultaneously.
\end{remark}
\begin{proof}[Proof of Proposition~\ref{prop1.2}]
Let us fix arbitrarily $M>0$ and assume that $\psi_0 \in H^s,\, \varphi \in C^r\left(\mathbb{T}^d ; \mathbb{R}\right)$, and $u \in \mathbb{R}^q$ are such that
\begin{equation}\label{(3.1)}
	\left\|\psi_0\right\|_s+\|\varphi\|_{C^r}+\|u\|_{\mathbb{R}^q} \leq M.
\end{equation}
Recall that $$ r_1+ir_2=(1+i\nu)(a+ib)^2. $$
For any $\delta>0$, we denote $\phi(t):=e^{(a+ib) \delta^{-1 / 2} \varphi} \mathcal{R}_t\left(e^{-(a+ib) \delta^{-1 / 2} \varphi} \psi_0, (\delta^{-1} u,r_1,r_2)\right)$. According to Proposition~\ref{prop1.1}, $\phi(t)$ exists up to some maximal time $T^\delta=\tilde{T}\left(e^{-(a+ib) \delta^{-1 / 2} \varphi} \psi_0, \delta^{-1} u\right)$. Moreover, if $T^\delta<\infty$, then
$$
\left\|e^{-(a+ib) \delta^{-1 / 2} \varphi} \phi(t)\right\|_s \rightarrow \infty \quad \text { as } t \rightarrow T^{\delta-}.
$$

We need to show that
\begin{itemize}
	\item[(a)] there is a constant $\delta_0>0$ such that $T^\delta>\delta$ for any $\delta<\delta_0$;
	\item[(b)] $\phi(\delta) \rightarrow e^{(r_1+i r_2)\left( \mathbb{B}(\varphi)+\langle u, Q\rangle\right)} \psi_0$ in $H^s$ as $\delta \rightarrow 0^{+}$.
\end{itemize}
To prove these properties, we introduce the functions
\begin{equation}
	w(t):=e^{(1+i\nu)(a+ib)^2\left( \mathbb{B}(\varphi)+\langle u, Q\rangle\right) t} \psi_0^\delta=e^{(r_1+i r_2)\left( \mathbb{B}(\varphi)+\langle u, Q\rangle\right) t} \psi_0^\delta, \quad v(t):=\phi(\delta t)-w(t),
\end{equation}
where $\psi_0^\delta \in H^r$ satisfies\footnote{In what follows, $C$ denotes positive constants which may change from line to line. These constants depend on the parameters $M,\, V,\, Q, \,\nu,\, \mu,\, r_1,\, r_2,\, a,\, b,\, \sigma,\, d,\, s$, but not on $\delta$.}
\begin{align}
	&\|\psi_0^\delta\|_{s}  \leq C  \quad\quad\quad  \quad \text{for } \delta\leq 1,\label{(3.3)}\\
	& \|\psi_0^\delta\|_r \leq C   \delta^{-1/4}  \,\,\,\,\quad \text{for } \delta\leq 1,\label{(3.4)}\\
	&  \|	\psi_0-\psi_0^\delta\|_s\to 0  \quad\,\,\, \text{ as $\delta\to 0^+$}.\nonumber
\end{align}
For example, we could choose $\psi_0^\delta$ by using the heat semigroup: $\psi_0^\delta=e^{\delta^{1 / 4} \Delta} \psi_0$. In view of \eqref{(3.1)}-\eqref{(3.4)}, we have
\begin{align}
	\|w(t)\|_s &\leq C,\quad\,\,\quad\quad \forall t \in[0,2], \label{(3.5)}\\
	\|w(t)\|_r &\leq C \delta^{-1 / 4}, \quad \forall t \in[0,2] .\label{(3.6)}
\end{align}
Furthermore, $v(t)$ is well-defined for $t<\delta^{-1} T^\delta$ and satisfies the equation
\begin{align}
	\partial_t v= &\delta V(v+w) + \delta (1+i\nu)\Delta(v+w) -\delta (1+i\mu)|e^{-(a+ib)\delta^{-1/2}\varphi}(v+w)|^{2\sigma}(v+w)\nonumber\\
	& -\delta^{1/2}(1+i\nu)(a+ib)\left[\Delta\varphi(v+w) + 2\nabla\varphi\cdot\nabla(v+w)\right]\nonumber \\
	& +(r_1+ir_2)\left[\langle u,Q\rangle  +\mathbb{B}(\varphi)\right]v, \label{(3.7)}
\end{align}
with the initial condition
\begin{equation}\label{(3.8)}
	v(0)=\psi_0-\psi_0^\delta.
\end{equation}

Let us start by assuming that $ \psi_0\in  H^{2s+2}$. It follows that $ \psi(t)\in H^{2s+2}$ and $ v(t)\in H^{2s+2}$ for every $ t\in(0,T^\delta) $.
Let the multi-index notation $\alpha=\left(\alpha_1, \ldots, \alpha_d\right) \in \mathbb{N}^d$ satisfying $|\alpha|=\left|\alpha_1\right|+\cdots+\left|\alpha_d\right| \leq s$. Taking first the scalar product of \eqref{(3.7)} with $\partial^{2 \alpha} v$ in $L^2$ and then integrating by parts, we obtain
\begin{align}
	&\frac{\d}{\d t}\int_{[0,2\pi]^d}|\partial^\alpha v|^2\d x \nonumber\\
	\leq &2\delta|V|\|v+w\|_s\|v\|_s \nonumber\\
	& +2 \delta  (1+\nu^2)^{1/2}\|w\|_r\|v\|_s\nonumber\\
	& +2 \delta (1+\mu^2)^{1/2}\left| \langle  \partial^\alpha[| e^{-(a+ib)\delta^{-1/2} \varphi}(v+w) |^{2 \sigma}(v+w)], \partial^{\alpha} v\rangle_{L^2([0,2\pi]^d)} \right|\nonumber\\
	& +2 \delta^{1 / 2} (1+\nu^2)^{1/2} (a^2+b^2)^{1/2} \left|  \langle\partial^\alpha[\Delta \varphi(v+w) + 2\nabla \varphi \cdot \nabla(v+w)], \partial^{\alpha} v \rangle_{L^2([0,2\pi]^d)}\right|\nonumber\\
	&  +2 (r_1^2+r_2^2)^{1/2}\left| \langle\partial^\alpha\left[(\langle u, Q\rangle+\mathbb{B}(\varphi)) v\right], \partial^{\alpha} v \rangle_{L^2([0,2\pi]^d)}\right|\nonumber\\
	=:&\sum_{i=1}^{5}I_i. \label{(3.9)}
\end{align}
We estimate the terms $I_1,\, I_2,\, I_3$, and $I_5$ by integrating by parts and by using \eqref{(3.1)}, \eqref{(3.5)}, and \eqref{(3.6)}:
$$
\begin{aligned}
	& \left|I_1\right| \leq C \delta\|v+w\|_s\|v\|_s \leq  C \delta\|v\|_s^2+C \delta\|v\|_s, \\
	& \left|I_2\right| \leq C \delta\|w\|_r\|v\|_s \leq C \delta^{3 / 4}\|v\|_s, \\
	& \left|I_3\right| \leq C \delta\|v+w\|_s^{2 \sigma+1}\|v\|_s \leq C \delta\|v\|_s^{2(\sigma+1)}+C \delta\|v\|_s \\
	& \left|I_5\right| \leq C\|v\|_s^2 .
\end{aligned}
$$
We estimate $I_4$ as follows:
$$
\left|I_4\right| \leq C \delta^{1 / 2}\|v\|_s^2+C \delta^{1 / 2}\|w\|_{s+1}\|v\|_s \leq C \delta^{1 / 2}\|v\|_s^2+C \delta^{1 / 4}\|v\|_s ,
$$
where we have used again integration by parts, the identities \eqref{(3.1)}, \eqref{(3.5)} and \eqref{(3.6)}, and the equality
$$
\left\langle\partial_{x_j} \varphi \partial_{x_j} \partial^\alpha v, \partial^\alpha v\right\rangle_{L^2}=\frac{1}{2} \langle\partial_{x_j} \varphi, \partial_{x_j}\left| \partial^\alpha v\right|^2 \rangle_{L^2}=- \langle\partial_{x_j}^2 \varphi,\left| \partial^\alpha v\right|^2 \rangle_{L^2}.
$$
Summing up inequalities \eqref{(3.9)} for all $\alpha \in \mathbb{N}^d,\,|\alpha| \leq s$, combining the resulting inequality with the estimates for $I_j$ and the Young inequality, and recalling that $\delta \leq 1$, we obtain
$$
\partial_t\|v\|_s^2 \leq C \delta^{1 / 2}+C\left(1+\delta^{1 / 2}\right)\|v\|_s^2+C \delta\|v\|_s^{2(\sigma+1)}, \quad t \leq \delta^{-1} T^\delta.
$$
This inequality, together with \eqref{(3.8)} and the Gronwall inequality, implies that
\begin{equation}\label{(3.10)}
	\|v(t)\|_s^2 \leq e^{C\left(1+\delta^{1 / 2}\right) t}\left(C \delta^{1 / 2} t+\left\|\psi_0-\psi_0^\delta\right\|_s^2+C \delta \int_0^t\|v(\tau)\|_s^{2(\sigma+1)} \d \tau\right)
\end{equation}
for $t \leq \delta^{-1} T^\delta$ and for every $ \psi_0\in  H^{2s+2}$.  

Finally, we can extend the validity of \eqref{(3.10)} for every $\psi_0 \in H^s$ thanks to (i) of Proposition~\ref{prop1.1} and to the density of $H^{2 s+2}$ into $H^s$ with respect to the $H^s$-norm.
Let us take $\delta_0 \in(0,1)$ small enough such that, for $\delta<\delta_0$,
\begin{align}
	& \left\|\psi_0-\psi_0^\delta\right\|_s^2<1, \label{(3.11)}\\
	& e^{C\left(1+\delta^{1 / 2}\right)}\left(C \delta^{1 / 2}+\left\|\psi_0-\psi_0^\delta\right\|_s^2\right)<1 / 2,\label{(3.12)}
\end{align}
and denote
$$
\tau^\delta=\sup \left\{t<\delta^{-1} T^\delta\mid\|v(t)\|_s<1\right\}.
$$
From \eqref{(3.8)} and \eqref{(3.11)} it follows that $\tau^\delta>0$ for $\delta<\delta_0$. Let us show that $\tau^\delta>1$, provided that
\begin{equation}\label{(3.13)}
	\delta_0<\left(2 C e^{2 C}\right)^{-1} .
\end{equation}
To reach a contradiction, we assume that $\tau^\delta \leq 1$. Let $t=\tau^\delta$ in \eqref{(3.10)}. By using \eqref{(3.12)} and \eqref{(3.13)}, we obtain
$$
1=\left\|v(\tau^\delta)\right\|_s^2<\frac{1}{2}+\frac{1}{2} \int_0^{\tau^\delta}\|v(y)\|_s^{2(\sigma+1)} \mathrm{d} y \leq 1 .
$$
This contradiction shows that $\tau^\delta>1$ for $\delta<\delta_0$. Therefore, we have $1<\delta^{-1} T^\delta$. Thus, the property (a) is proved. Taking $t=1$ in \eqref{(3.10)}, we arrive at
$$
\|v(1)\|_s^2 \leq e^{C\left(1+\delta^{1 / 2}\right)}\left(C \delta^{1 / 2}+\left\|\psi_0-\psi_0^\delta\right\|_s^2+C \delta\right) \rightarrow 0 \quad \text { as } \delta \rightarrow 0^{+} .
$$
This implies (b), and completes the proof.
\end{proof}

\section{Small-time approximate controllability}\label{sec:Approximate controllability}
In what follows, we assume that $s \geq s_d$ is an integer and denote $r=s+2$ as in Proposition~\ref{prop1.2}. 
We start this section with a definition of a saturation property introduced in \cite{Duca_EMS}. 
Let $H$ be a finite-dimensional subspace of $C^r\left(\mathbb{T}^d ; \mathbb{R}\right)$, and let $F(H)$ be the largest subspace of $C^r\left(\mathbb{T}^d ; \mathbb{R}\right)$ whose elements can be represented in the form
$$
\theta_0+\sum_{j=1}^n \mathbb{B}\left(\theta_j\right)
$$
for some integer $n \geq 1$ and functions $\theta_j \in H,\, j=0, \ldots, n$, where $\mathbb{B}$ is given by \eqref{(1.2)}. As $\mathbb{B}$ is quadratic, $F(H)$ is well-defined and finite-dimensional. 
Let us define a nondecreasing sequence $\left\{H_j\right\}$ of finite-dimensional subspaces by $H_0=H$ and $H_j=F\left(H_{j-1}\right), j \geq 1$, and denote
\begin{equation*}\label{(2.1)}
H_{\infty}:=\bigcup_{j=1}^{\infty} H_j .
\end{equation*}
\begin{definition}\label{def2.1}
A finite-dimensional subspace $H \subset C^r\left(\mathbb{T}^d ; \mathbb{R}\right)$ is said to be saturating, if $H_{\infty}$ is dense in $C^r\left(\mathbb{T}^d ; \mathbb{R}\right)$.
\end{definition}
Throughout this section, we assume that the following condition is satisfied:
\begin{description}\item[\hypertarget{H1}]{\textbf{($\mathscr{P}$)}}  
The field $Q=(Q_1, \ldots, Q_q) $ is saturating, i.e., the subspace $$H=\operatorname{span}\left\{Q_j\mid j=1, \ldots, q\right\}$$ is saturating in the sense of Definition~\ref{def2.1}.
\end{description}  
\subsection{Small-time approximate controllability}
A direct consequence of Proposition~\ref{prop1.2} is the small-time globally approximate null-controllability.
\begin{theorem}\label{thm00}
Assume that the condition~{\hyperlink{H1}{\textbf{($\mathscr{P}$)}}} is satisfied, $ \mathbf{1}\in \operatorname{span}\left\{Q_j\mid j=1, \ldots, q\right\} $ and $\psi_0 \in H^s$, where $s \geq s_d$ is an integer. For any $\varepsilon,\, T>0$ and $r_1,\, r_2\in\R$, there exists $u\in \Theta\left(\psi_0, T\right) $ such that 
\begin{equation*}
\|\mathcal{R}_T \left(\psi_0, (u,r_1,r_2)\right)\|_s< \varepsilon.
\end{equation*}
\end{theorem}
\begin{proof}
	For any $\varepsilon>0$, let $c\in\R$  satisfy
	\begin{equation*}
		|e^{(r_1+ir_2)c}|=e^{cr_1}<\frac{\varepsilon}{2\|\psi_0\|_s}.
	\end{equation*}
	That is
	\begin{equation*}
		\|e^{(r_1+ir_2)c}\psi_0\|_s<\frac\varepsilon2.
	\end{equation*}
	Since $ \mathbf{1}\in \operatorname{span}\left\{Q_j\mid j=1, \ldots, q\right\} $, we can choose $u_c\in\R^q$ such that  $c=\langle u_c,Q\rangle$.
	Thanks to Proposition~\ref{prop1.2}, there exists $\delta>0$ such that $\delta^{-1} u_c \in \Theta\left( \psi_0, \delta\right)$ and
	\begin{equation*}
		\|\mathcal{R}_\delta\left(  \psi_0, (\delta^{-1} u_c,r_1,r_2)\right)- e^{(r_1+i r_2)c} \psi_0\|_s< \frac{\varepsilon}{2}.
	\end{equation*}
	Hence
	\begin{equation*}
		\|\mathcal{R}_\delta\left(  \psi_0, (\delta^{-1} u_c,r_1,r_2)\right) \|_s\leq \|\mathcal{R}_\delta\left(  \psi_0, (\delta^{-1} u_c,r_1,r_2)\right)- e^{(r_1+i r_2)c} \psi_0\|_s +\|e^{(r_1+i r_2)c} \psi_0\|_s< \varepsilon.
	\end{equation*}
	By (i) of Proposition~\ref{prop1.1}, $\mathcal{R}_{t}\left(\mathcal{R}_\delta\left(  \psi_0, (\delta^{-1} u_c,r_1,r_2)\right), (0,r_1,r_2)\right)$ and $\mathcal{R}_{t}\left(0,(0,r_1,r_2)\right)\equiv0$ are defined in the same time interval $t\in[0,T-\delta]$ when $\varepsilon$ is sufficiently small.
	Furthermore, (ii) of Proposition~\ref{prop1.1} implies  the existence of $C > 0$, independent of $ \mathcal{R}_\delta\left(  \psi_0, (\delta^{-1} u_c,r_1,r_2)\right)$, such that
	\begin{equation*}
		\|\mathcal{R}_{T-\delta}\left(\mathcal{R}_\delta\left(  \psi_0, (\delta^{-1} u_c,r_1,r_2)\right), (0,r_1,r_2)\right)- \mathcal{R}_{T-\delta}\left(0,(0,r_1,r_2)\right)\|_s\leq C 	\|\mathcal{R}_\delta\left(  \psi_0, (\delta^{-1} u_c,r_1,r_2)\right) \|_s <C\varepsilon.
	\end{equation*}
	Hence, the proof is completed by taking the control \begin{equation*}
		u(t)= \begin{cases}
		\delta^{-1}u_c,\quad &t\in(0,\delta),\\
		0,\quad &t\in(\delta,T). 
		\end{cases} 
	\end{equation*}
\end{proof}

\begin{theorem}\label{thm2.2}
Assume that the condition~{\hyperlink{H1}{\textbf{($\mathscr{P}$)}}} is satisfied, $ \mathbf{1}\in \operatorname{span}\left\{Q_j\mid j=1, \ldots, q\right\} $ and $ r_1=1,\, r_2=-\nu $.
Then for any $\varepsilon>0,\, \mathcal{T}>0,\, \psi_0 \in H^s$, where $s \geq s_d$ is an integer, and $\theta \in C^r\left(\mathbb{T}^d ; \mathbb{R}\right)$, there are $T \in(0, \mathcal{T})$ and $u \in \Theta\left(\psi_0, T\right) \cap C^{\infty}\left(J_T ; \mathbb{R}^q\right)$ such that
\begin{equation}\label{(2.2)}
\left\|\mathcal{R}_T\left(\psi_0, (u,1,-\nu)\right)-e^{(1-i\nu) \theta} \psi_0\right\|_s<\varepsilon.
\end{equation}
\end{theorem}
\begin{proof}
We first let $ r_1=a=\frac{1}{1+\nu^2},\, r_2 = b =\frac{-\nu}{1+\nu^2} $, and
use an induction argument in $N=0,1,2,...$ to show that the approximate controllability in this theorem is true for any $\theta \in H_N$. 
More precisely, we prove the following property:
\begin{description}\item[\hypertarget{PN}]{\textbf{(P$_N$)}}  
For any $\theta \in H_N$ and $\psi_0 \in H^s$, there is a family $\left\{u_\tau\right\}_{\tau>0} \subset L^2\left(J_1 ; \mathbb{R}^q\right)$ such that $u_\tau \in \Theta\left(\psi_0, \tau\right)$ for sufficiently small $\tau>0$, and
\begin{equation}\label{(2.3)}
\mathcal{R}_\tau\left(\psi_0, \left(u_\tau,\frac{1}{1+\nu^2},\frac{-\nu}{1+\nu^2}\right)\right) \rightarrow e^{\frac{1-i\nu}{1+\nu^2} \theta} \psi_0 \quad \text { in } H^s \text { as } \tau \rightarrow 0^{+}.
\end{equation}
\end{description} 
Combined with the saturation condition~{\hyperlink{H1}{\textbf{($\mathscr{P}$)}}}, this leads to the approximate controllability for any $\theta \in C^r\left(\mathbb{T}^d ; \mathbb{R}\right)$.
\\
\textbf{Step~1.} Case $N=0$. Applying Proposition~\ref{prop1.2} for $\varphi=0$ and $u \in \mathbb{R}^q$ with $\theta=$ $\langle u, Q\rangle$, we obtain
$$
\mathcal{R}_\delta\left(\psi_0, \left(\delta^{-1} u,\frac{1}{1+\nu^2},\frac{-\nu}{1+\nu^2}\right)\right) \rightarrow e^{\frac{1-i\nu}{1+\nu^2} \theta } \psi_0 \quad \text { in } H^s \text { as } \delta \rightarrow 0^{+}.
$$
This implies \eqref{(2.3)} with $\tau=\delta$ and $u_\tau=\delta^{-1} u$.
\\
\textbf{Step~2.} Case $N \geq 1$. We assume that \hyperlink{PN}{\textbf{(P$_{N-1}$)}} is true. Let $\tilde{\theta} \in H_N$ be of the form
$$
\tilde{\theta}=\theta_0+\sum_{j=1}^n \mathbb{B}\left(\theta_j\right),
$$
where $n \geq 1$ and $\theta_j \in H_{N-1},\, j=0, \ldots, n$.
Take $ \theta_1 $ and $ c>0 $ to be such that $ \tilde{\theta}_1=\theta_1 +c\geq0 $. Note that $ \mathbb{B}(\tilde{\theta}_1) = \mathbb{B}(\theta_1) $. Applying Proposition~\ref{prop1.2} with $\varphi=\tilde{\theta}_1$ and $u=0$, we get
\begin{equation*}\label{(2.4)}
e^{\frac{1-i\nu}{1+\nu^2}\delta^{-1 / 2} \tilde{\theta}_1} \mathcal{R}_\delta\left(e^{-\frac{1-i\nu}{1+\nu^2} \delta^{-1 / 2} \tilde{\theta}_1} \psi_0, \left(0,\frac{1}{1+\nu^2},\frac{-\nu}{1+\nu^2}\right)\right) \rightarrow e^{\frac{1-i\nu}{1+\nu^2} \mathbb{B}\left(\theta_1\right)} \psi_0 \quad \text { in } H^s \text { as } \delta \rightarrow 0^{+} .
\end{equation*}
Since $ c\in H_0$ and $\theta_1\in H_{N-1} $, we have $ \tilde{\theta}_1\in H_{N-1} $. 
The induction hypothesis implies that, for any $\delta>0$, there are families of controls $\left\{u_{\tau, \delta}^1\right\} \subset L^2\left(J_1 ; \mathbb{R}^q\right)$ and $\left\{u_{\tau, \delta}^2\right\} \subset L^2\left(J_1 ; \mathbb{R}^q\right)$ such that $u_{\tau, \delta}^1 \in \Theta\left(\psi_0, \tau\right)$ and $u_{\tau, \delta}^2 \in \Theta\left(\mathcal{R}_\delta\left(e^{-\frac{1-i\nu}{1+\nu^2} \delta^{-1 / 2} \tilde{\theta}_1} \psi_0, \left(0,\frac{1}{1+\nu^2},\frac{-\nu}{1+\nu^2}\right)\right), \tau\right)$ for sufficiently small $\tau>0$, and
\begin{equation*}\label{(2.5)}
\mathcal{R}_\tau\left(\psi_0, \left(u_{\tau, \delta}^1,\frac{1}{1+\nu^2},\frac{-\nu}{1+\nu^2}\right)\right) \rightarrow e^{-\frac{1-i\nu}{1+\nu^2} \delta^{-1 / 2} \tilde{\theta}_1} \psi_0,
\end{equation*}
\begin{align}
&\mathcal{R}_\tau\left(\mathcal{R}_\delta\left(e^{-\frac{1-i\nu}{1+\nu^2} \delta^{-1 / 2} \tilde{\theta}_1} \psi_0, \left(0,\frac{1}{1+\nu^2},\frac{-\nu}{1+\nu^2}\right)\right), \left(u_{\tau, \delta}^2,\frac{1}{1+\nu^2},\frac{-\nu}{1+\nu^2}\right)\right) \nonumber\\
 &\rightarrow e^{\frac{1-i\nu}{1+\nu^2} \delta^{-1 / 2} \tilde{\theta}_1} \mathcal{R}_\delta\left(e^{-\frac{1-i\nu}{1+\nu^2} \delta^{-1 / 2} \tilde{\theta}_1} \psi_0, \left(0,\frac{1}{1+\nu^2},\frac{-\nu}{1+\nu^2}\right)\right) \nonumber\label{(2.6)}
\end{align}
in $H^s$ as $\tau \rightarrow 0^{+}$. 

Combining these controls and using Proposition~\ref{prop1.1}, we can construct a new family $\left\{u_\tau^1\right\} \subset L^2\left(J_1 ; \mathbb{R}^q\right)$ such that $u_\tau^1 \in \Theta\left(\psi_0, \tau\right)$ for sufficiently small $\tau>0$, and
$$
\mathcal{R}_\tau\left(\psi_0, \left(u_\tau^1,\frac{1}{1+\nu^2},\frac{-\nu}{1+\nu^2}\right)\right) \rightarrow e^{\frac{1-i\nu}{1+\nu^2} \mathbb{B} (\theta_1 )} \psi_0 \quad \text { in } H^s \text { as } \tau \rightarrow 0^{+} .
$$
Iterating this argument with $\theta_j \in H_{N-1},\, j=0, \ldots, n$, we obtain a family $\left\{u_\tau^n\right\} \subset$ $L^2\left(J_1 ; \mathbb{R}^q\right)$ such that $u_\tau^n \in \Theta\left(\psi_0, \tau\right)$ for small $\tau>0$ and
$$
\mathcal{R}_\tau\left(\psi_0, \left(u_\tau^n,\frac{1}{1+\nu^2},\frac{-\nu}{1+\nu^2}\right)\right) \rightarrow e^{\frac{1-i\nu}{1+\nu^2}\left(\theta_0+\sum_{j=1}^n \mathbb{B} (\theta_j )\right)} \psi_0=e^{\frac{1-i\nu}{1+\nu^2} \tilde{\theta}} \psi_0 \quad \text { in } H^s \text { as } \tau \rightarrow 0^{+} .
$$
As $\tilde{\theta} \in H_N$ is arbitrary, this showes the required Property~\hyperlink{PN}{\textbf{(P$_{N}$)}} for $N$.
\\
\textbf{Step~3. Conclusion.}  Finally, let $\theta \in C^r\left(\mathbb{T}^d ; \mathbb{R}\right)$ be arbitrary. It is clear that $ (1+\nu^2)\theta\in C^r\left(\mathbb{T}^d ; \mathbb{R}\right) $. By the saturation hypothesis~{\hyperlink{H1}{\textbf{($\mathscr{P}$)}}}, $H_{\infty}$ is dense in $C^r\left(\mathbb{T}^d ; \mathbb{R}\right)$. This implies that we can find $N \geq 1$ and $\tilde{\theta} \in H_N$ such that
$$
\left\|e^{(1-i\nu)\theta} \psi_0-e^{\frac{1-i\nu}{1+\nu^2} \tilde{\theta}} \psi_0\right\|_s<\varepsilon.
$$
Applying \hyperlink{PN}{\textbf{(P$_{N}$)}}  for $\tilde{\theta} \in H_N$, we find $T \in(0, \mathcal{T})$ and $\tilde{u} \in \Theta\left(\psi_0, T\right)$ such that \begin{equation*}
	\left\|\mathcal{R}_T\left(\psi_0, \left(\tilde{u},\frac{1}{1+\nu^2},\frac{-\nu}{1+\nu^2}\right)\right)-e^{(1-i\nu) \theta} \psi_0\right\|_s<2\varepsilon.
\end{equation*}
Since
\begin{equation*}
\mathcal{R}_T\left(\psi_0, \left(\tilde{u},\frac{1}{1+\nu^2},\frac{-\nu}{1+\nu^2}\right)\right) = \mathcal{R}_T\left(\psi_0, \left(\frac{\tilde{u}}{1+\nu^2},1,-\nu\right)\right),
\end{equation*}
we obtain \eqref{(2.2)} with $ u= \frac{\tilde{u}}{1+\nu^2} $.

Proposition~\ref{prop1.1} and a density argument show that we can take $u \in \Theta\left(\psi_0, T\right) \cap C^{\infty}\left(J_T ; \mathbb{R}^q\right)$. Indeed, since $C^{\infty}\left(J_T ; \mathbb{R}^q\right)$ is dense in $L^2\left(J_T ; \mathbb{R}^q\right)$, for any $u\in  \Theta\left(\psi_0, T\right) $ such that  \eqref{(2.2)} is satisfied, one can choose $u_\varepsilon\in C^{\infty}\left(J_T ; \mathbb{R}^q\right)$ satisfying
\begin{equation*}
	\|u_\varepsilon - u\|_{L^2\left(J_T ; \mathbb{R}^q\right)}\leq \varepsilon,
\end{equation*}
this, along with \eqref{(2.2)} and  Proposition~\ref{prop1.1}, finally implies that $u_\varepsilon\in  \Theta\left(\psi_0, T\right) $ and
\begin{equation*}
	\begin{aligned}
	&\left\|\mathcal{R}_T\left(\psi_0, (u_\varepsilon,1,-\nu)\right)-e^{(1-i\nu) \theta} \psi_0\right\|_s \\
	\leq 	&\left\|\mathcal{R}_T\left(\psi_0, (u_\varepsilon,1,-\nu)\right)- \mathcal{R}_T\left(\psi_0, (u,1,-\nu)\right)\right\|_s + \left\|\mathcal{R}_T\left(\psi_0, (u,1,-\nu)\right)-e^{(1-i\nu) \theta} \psi_0\right\|_s\\
	< &( C+1 )\varepsilon.
	\end{aligned}
\end{equation*}
\end{proof}
\begin{remark}
Under the conditions of Theorem~\ref{thm2.2}, for any $M>0, \,\mathcal{T}>0$, and nonzero $\psi_0 \in H^s$, there exist a time $T \in(0, \mathcal{T})$ and a control $u \in \Theta\left(\psi_0, T\right)$ such that
$$
\left\|\mathcal{R}_T\left(\psi_0, \left(u,1,-\nu\right)\right)\right\|_s>M.
$$
It suffices to apply Theorem~\ref{thm2.2} by choosing $\theta \in C^r\left(\mathbb{T}^d ; \mathbb{R}\right)$ such that
$$
\left\|e^{(1-i\nu) \theta} \psi_0\right\|_s>M.
$$
To this aim, one can take arbitrarily $\theta_1 \in C^r\left(\mathbb{T}^d ; \mathbb{R}\right)$ satisfying $\left\|e^{(1-i\nu) \theta_1} \psi_0\right\|_1 \neq 0$, and put $\theta=\lambda \theta_1$ with sufficiently large $\lambda>0$, and use the inequality $\|\cdot\|_1 \leq\|\cdot\|_s$.
\end{remark}
\begin{remark}
Let $ \nu=0$ and the conditions of Theorem~\ref{thm2.2} be satisfied, then \eqref{(0.1)} is small-time global approximately controllable in $L^2$ between states that share the same argument, see Subsection~\ref{sec:reachable} for more details.
\end{remark}

\subsection{Proofs of Theorems~\ref{thm0} and \ref{thma}}
Let us end this section with an example of a saturating subspace. Let $I \subset \mathbb{Z}_*^d$ be a finite set and let
\begin{equation*}\label{(2.11)}
H(I)=\operatorname{span}\{\mathbf{1}, \sin \langle x, k\rangle, \cos \langle x, k\rangle\mid k \in I\} .
\end{equation*}
Recall that $I$ is a generator if any vector of $\mathbb{Z}^d$ is a linear combination of vectors of $I$ with integer coefficients. We write $m \perp l$ when the vectors $m,\, l \in \mathbb{R}^d$ are orthogonal and $m \not\perp l$ when they are not.
The following result is quoted from \cite[Proposition 2.6]{Duca_EMS}.
\begin{proposition}\label{prop2.6}
The subspace $H(I)$ is saturating in the sense of Definition~\ref{def2.1}, if and only if $I$ is a generator and for any $l,\, m \in I$, there are vectors $\left\{n_j\right\}_{j=1}^\sigma \subset I$ such that $l \not \perp n_1,\, n_j \not \perp n_{j+1}$ for $j=1, \ldots, \sigma-1$, and $n_\sigma \not \perp m$.
\end{proposition}
\begin{proof}[Proofs of Theorems~\ref{thm0} and \ref{thma}]
	Clearly, the set $K \subset \mathbb{Z}_*^d$ defined by \eqref{(0.3)} satisfies the condition in  Proposition~\ref{prop2.6}. Therefore, the subspace $H(K)$ is saturating, and Theorems~\ref{thm0}, \ref{thma}
	follow immediately from Theorems~\ref{thm00}, \ref{thm2.2}, respectively.
\end{proof}

\section{Concluding comments}\label{sec:conclud}
\subsection{Controllability on bounded domain}
It is interesting to notice that most controllability results are established on the torus, while only a few have been established for bounded domains. We refer to \cite[Section~6 and Section~7]{Duca_Burgers} for bilinear small-time controllability results of the Burgers equation with Dirichlet and Neumann boundary conditions. It would be interesting to extend these bilinear small-time controllability results to other evolution equations.
\subsection{Reachable subspace}\label{sec:reachable}
Let $ \nu=0$ and the conditions of Theorem~\ref{thm2.2} be satisfied, then \eqref{(0.1)} becomes a complex heat equation with complex nonlinearity. For any $\psi_0,\,\psi_1\in H^s$ with the same argument, i.e., $ Arg(\psi_0)=Arg(\psi_1)$, where $s \geq s_d$ is an integer, we can prove that: for any $\varepsilon>0,\, \mathcal{T}>0$, there are $T \in(0, \mathcal{T})$ and $u \in \Theta\left(\psi_0, T\right) \cap C^{\infty}\left(J_T ; \mathbb{R}^q\right)$ such that
\begin{equation}\label{eq:globalArg}
	\left\|\mathcal{R}_T\left(\psi_0, (u,1,0)\right)-  \psi_1\right\|_{L^2}<\varepsilon,
\end{equation}i.e., the small-time global approximately controllability is satisfied in $L^2$ between states that share the same argument.

Indeed, if we denote by $Z$ the set of zeroes of $\psi_0$ and $\psi_1$, then the fact that $Arg(\psi_0)=Arg(\psi_1)$ implies that $$\psi_0=|\psi_0|e^{i Arg(\psi_0)}=\frac{|\psi_0|}{|\psi_1|}|\psi_1|e^{iArg(\psi_1)}=\frac{|\psi_0|}{|\psi_1|}\psi_1 \text{ on } Z.$$
Consider for $\eta>0$ the set
\begin{equation*}
	Z_\eta:=\left\{x \in \mathbb{T}^d \mid \operatorname{dist}(x, Z)<\eta\right\},
\end{equation*}
and $Z_\eta^c:=\mathbb{T}^d\setminus Z_\eta$.   For $\eta>0$, we define
\begin{equation*}
	\phi_\eta=\rho_\eta \ln \left(\psi_1 / \psi_0\right)=\rho_\eta \ln \left(|\psi_1| / |\psi_0|\right),
\end{equation*}
where \begin{equation*}
	\rho_\eta=\left\{\begin{aligned}
		&1,\quad &&x\in Z_{2\eta}^c
		\\&(0,1),\quad && x\in Z_\eta^c\setminus Z_{2\eta}^c
		\\&0.\quad && x\in Z_\eta
	\end{aligned}\right.
\end{equation*}  is a mollifier compactly supported inside $Z_\eta^c$. $\phi_\eta$ is well-defined because $|\psi_1| / |\psi_0|>0$ on $Z_\eta^c$. Furthermore, $\phi_\eta$ belongs to $H^s\left(\mathbb{T}^d;\mathbb{R}\right)$. Notice that
\begin{equation*}
	\left\|e^{ \phi_\eta} \psi_0-\psi_1\right\|_{L^2\left(\mathbb{T}^d\right)}
	\leq\left\|e^{ \phi_\eta} \psi_0-\psi_1\right\|_{L^2\left(Z_{2\eta}\right)}.
\end{equation*}
Hence, for any $\varepsilon,\, T>0$, we can choose $\eta>0$ small enough such that
\begin{equation*}
	\left\|e^{\phi_\eta} \psi_0-\psi_1\right\|_{L^2\left(\mathbb{T}^d\right)}<\varepsilon / 3 .
\end{equation*}
Now, by the density of $C^{s+2}\left(\mathbb{T}^d;\R\right)$ into $H^s\left(\mathbb{T}^d;\R\right)$, there exists $\widetilde{\phi}_\eta \in C^{s+2}\left(\mathbb{T}^d;\R\right)$ such that
\begin{equation*}
	\left\|e^{\widetilde{\phi}_\eta} \psi_0-\psi_1\right\|_{L^2\left(\mathbb{T}^d\right)} \leq\left\|e^{\widetilde{\phi}_\eta} \psi_0-e^{\phi_\eta} \psi_0\right\|_{L^2\left(\mathbb{T}^d\right)}+\left\|e^{\phi_\eta} \psi_0-\psi_1\right\|_{L^2\left(\mathbb{T}^d\right)}<\frac{2 \varepsilon}{3} .
\end{equation*}
Then, by applying Theorem~\ref{thm2.2} with $\theta=\widetilde{\phi}_\eta$, there are $T \in(0, \mathcal{T})$ and $u \in \Theta\left(\psi_0, T\right) \cap C^{\infty}\left(J_T ; \mathbb{R}^q\right)$ such that
\begin{equation*} 
	\left\|\mathcal{R}_T\left(\psi_0, (u,1,0)\right)-e^{ \theta} \psi_0\right\|_s<\frac{\varepsilon}{3}.
\end{equation*}
The triangular inequality finally implies that
\begin{equation*}
	\left\|\mathcal{R}_T\left(\psi_0, (u,1,0)\right)-  \psi_1\right\|_{L^2\left(\mathbb{T}^d\right)} \leq \left\|\mathcal{R}_T\left(\psi_0, (u,1,0)\right)-e^{ \widetilde{\phi}_\eta} \psi_0\right\|_{L^2\left(\mathbb{T}^d\right)}+\left\|e^{\widetilde{\phi}_\eta} \psi_0-\psi_1\right\|_{L^2\left(\mathbb{T}^d\right)}<\varepsilon,
\end{equation*}this completes the proof of \eqref{eq:globalArg}.

However, the characteristic of the reachable subspace in $H^s$ of the bilinear CGL equation remains open. 

\newpage

\noindent\textbf{Appendix}
\begin{proof}[Proof of Proposition~\ref{prop1.1}]
	For $V\geq0$, $\nu,\, \mu,\, r_1,\, r_2\in \R,\, Q\in C^\infty( \mathbb{T}^d ; \mathbb{R}^q)$ and $ u\in L^2_{\text{loc}}(\R^+;\R^{q})$, let
	\begin{equation*}
		L \psi:=V \psi+(1+i \nu) \Delta \psi, \quad N(t,u,\psi):=-(1+i \mu)|\psi|^{2 \sigma} \psi +(r_1+ir_2)\langle u(t),Q\rangle\psi .
	\end{equation*}
	The associated semigroup $S(t)$ of $ L $ acting on a Banach space $  \mathbb{X}$ can be written as a convolution: $S(t) \psi=G_t * \psi,\,\forall\psi\in \mathbb{X}$, with its Green function $G_t=G_t(x)$ for $t>0$ given by (see \cite{LO1994})
	\begin{align}
		G_t(x) & =(2\pi)^d\sum_{n \in \mathbb{Z}^d} g_t(x+2\pi n), \label{(5.5)}\\
		g_t(x) & =\frac{1}{(4 \pi(1+i \nu) t)^{d / 2}} \exp \left(-\frac{|x|^2}{4(1+i \nu) t}+V t\right).\nonumber
	\end{align}
	For $t>0$, the Green function \eqref{(5.5)} satisfies the $L^1$-estimate
	\begin{equation*}
		\|G_t \|_{L^1} \leq \sum_{n \in \mathbb{Z}^d} \int_{[0,2\pi]^d} |g_t(x+2\pi n) | \d x= \int_{\mathbb{R}^d} |g_t(x) | \d x=  (1+\nu^2 )^{d / 4} e^{V t},
	\end{equation*}
	from which it follows that $S(t)$ is bounded over $L^p$ for every $1 \leq p \leq \infty$ with
	\begin{equation*}
		\|S(t) \psi\|_{L^p}=\left\|G_t * \psi\right\|_{L^p} \leq  \left\|G_t\right\|_{L^1}\|\psi\|_{L^p} \leq  \left(1+\nu^2\right)^{d / 4} e^{V t}\|\psi\|_{L^p}.
	\end{equation*}
	Moreover, it can be shown (see \cite{LO1994}) that $S(t)$ is a strongly continuous semigroup over $C^0$ and over $L^p$ for every $1 \leq p<\infty$.
	\\
	\textbf{Step~1. $S(t)$ is a strongly continuous semigroup over $ H^s$.}
	\\
	First of all, it is direct to check that $$ S(t)S(\tau)=S(t+\tau),\, \forall t,\tau\geq 0 \text{ and } S(0)=I. $$ 
	Next, for any $ \psi\in H^s $, we have by the Fourier series expansion,
	\begin{equation*}
		\psi(x)=\sum_{k \in \mathbb{Z}^d} \widehat{\psi}(k) e^{i\langle k,x\rangle},
	\end{equation*}where $$ \widehat{\psi}(k)= \frac{1}{(2 \pi)^d} \int_{[0,2\pi]^d} \psi(x) e^{-i \langle k , x\rangle} \d x. $$ Moreover $$ \widehat{S(t)\psi}(k)=\widehat{G_t * \psi}(k)=\widehat{G_t}(k)\cdot \widehat{\psi}(k) =  e^{-(1+i \nu)|k|^2 t+V t} \widehat{\psi}(k).\footnote{Recall that $(f*g)(x)=\int_{\mathbb{T}^d}f(x-y)g(y)\d \mathfrak{m}(y),\,\forall x\in\mathbb{T}^d$ and $ \widehat{G_t}(k)= \int_{\R^d}g_t(x) e^{-i \langle k, x\rangle} \d x= \frac{e^{V t}}{(4 \pi(1+i \nu) t)^{d / 2}} \int_{\R^d}  e^{-\frac{|x|^2}{4(1+i \nu) t}-i \langle k, x\rangle} \d x=  e^{-(1+i \nu)|k|^2 t+V t}$.}$$ Hence
	\begin{equation*}
		\|S(t) \psi-\psi\|_{s}^2= \sum_{k \in \mathbb{Z}^d}\left(1+|k|^2\right)^s |   e^{-(1+i \nu)|k|^2 t+V t}-1 |^2 |\widehat{\psi}(k) |^2 .
	\end{equation*}The fact that $ e^{-(1+i \nu)|k|^2 t+V t}\rightarrow 1 \text{ as } t\rightarrow 0^+  $, along with the dominated convergence theorem, implies that 
	\begin{equation*}
		\|S(t) \psi-\psi\|_{s}^2\rightarrow 0,\quad \text{as }t\rightarrow 0^+.
	\end{equation*}
	Finally, notice that for any $ t\geq 0 $ and for any $ \psi\in H^s $,
	\begin{equation*}
		\begin{aligned}
			\|S(t) \psi\|_{s}^2&=  \sum_{k \in \mathbb{Z}^d}\left(1+|k|^2\right)^s |e^{-(1+i \nu)|k|^2 t+V t} |^2 |\widehat{\psi}(k)|^2 \\
			&=  \sum_{k \in \mathbb{Z}^d}\left(1+|k|^2\right)^s e^{(-2|k|^2 +2V )t} |\widehat{\psi}(k)|^2 \\
			&\leq  e^{2Vt} \sum_{k \in \mathbb{Z}^d}\left(1+|k|^2\right)^s  |\widehat{\psi}(k) |^2 =    e^{2Vt} \|\psi\|_s^2,
		\end{aligned}
	\end{equation*} hence $ S(t)\in \mathcal{L}(H^s;H^s),\, \forall t\geq 0 $. In conclusion,   $S(t)$ is a strongly continuous semigroup over $H^s$.
	\\
	\textbf{Step~2. For any $s>d / 2,\, \tilde{\psi}_0 \in H^s$, and $\tilde{u} \in L_{\text{loc}}^2\left(\mathbb{R}_{+} ; \mathbb{R}^q\right)$, there exists $ t_1>0 $ such that the problem~\eqref{(0.1)}, \eqref{(0.2)} admits a unique mild solution $ \tilde{\psi}\in C\left([0,t_1];H^s\right) $ in the following form:
		\begin{equation}\label{(5.2)}
			\tilde{\psi}(t)= S(t)\tilde{\psi}_0 + \int_0^t S(t-\tau) N(\tau,\tilde{u},\tilde{\psi}(\tau))\d\tau.
	\end{equation}}
	\\
	The integral equation \eqref{(5.2)} recasts in terms of this Green function of form
	\begin{equation*}
		\tilde{\psi}(t)=G_t * \tilde{\psi}_0+\int_0^t G_{t-\tau} * N (\tau,\tilde{u},\tilde{\psi}\left(\tau\right) ) \d \tau.
	\end{equation*}
	Since $s>d / 2$, we deduce that the embedding $H^s \hookrightarrow C^0$ is continuous, namely, there exists a constant $C\left(s,d\right)>0$ such that
	\begin{equation*}
		\sup _{x \in \mathbb{T}^d}|y(x)| \leq C\left(s,d\right)\|y\|_{s}, \quad \forall y \in H^s.
	\end{equation*}
	Moreover, $H^s$ is a Banach algebra (see \cite[Theorem~4.39]{Adams}), i.e., there exists a constant $C\left(s,d\right)>0$ such that
	\begin{equation}\label{eq_BG}
		\|fg\|_s\leq C\left(s,d\right) \|f\|_s \|g\|_s,\quad \forall f, g\in H^s.
	\end{equation}
	For our later use, we define the following quantities
	\begin{equation*}
		\begin{aligned}
			&M:=\sup \left\{\left\|S(t)-I\right\|_{\mathcal{L}\left(H^s; H^s\right)}\mid0 \leq t \leq 1\right\}, \\ 
			&r (\tilde{\psi}_0 ):=2 M \|\tilde{\psi}_0 \|_{s}, \\
			&C(Q):=C(s,d)\max _{1 \leq i \leq q}\left\|Q_i\right\|_{s}, \\ 
			&C(q,r_1,r_2, \tilde{u}, Q):=(M+1) \sqrt{2q( r_1^2+r_2^2)}  C(Q) \|\tilde{u}\|_{L^2(0,1)},\\
			&\mathcal{T}_1 := C(q,r_1,r_2, \tilde{u}, Q)\left(r (\tilde{\psi}_0 )+ \|\tilde{\psi}_0 \|_{s}\right),\\
			&\mathcal{T}_2 := (M+1)(1+\mu^2)^{1/2}\left(C\left(s,d\right) 2\right)^{2\sigma}\left(r(\tilde{\psi}_0)^{2\sigma+1} + \|\tilde{\psi}_0\|_s^{2\sigma+1} \right).
		\end{aligned}
	\end{equation*}
	We now define
	\begin{equation*}
		\mathcal{T}:=\min \left\{1,\quad \frac{1}{4}\left(\frac{r (\tilde{\psi}_0 )}{\mathcal{T}_1+\mathcal{T}_2}\right)^2\right\},
	\end{equation*}
	and set $t_1=\mathcal{T}$.
	We denote $B:=B_{ C\left( [0, t_1 ]; H^s\right)}\left(\tilde{\psi}_0, r (\tilde{\psi}_0 )\right)$ the ball in the space $C\left( [0, t_1 ]; H^s\right)$ of center $\tilde{\psi}_0$ and radius $r (\tilde{\psi}_0 )$. For every $\psi \in B$ we define the following function
	\begin{equation*}
		\Phi(\psi)(t):=S(t) \tilde{\psi}_0+\int_0^t S(t-\tau)\left[-(1+i \mu)|\psi(\tau)|^{2 \sigma} \psi(\tau) +(r_1+ir_2)\langle \tilde{u}(\tau),Q\rangle\psi(\tau)\right] \d \tau.
	\end{equation*}We will finish \textbf{Step~2} by the following \textbf{Substep~2.1-- Substep~2.3}.
	\\
	\textbf{Substep~2.1. $\Phi$ maps $B$ into itself.}
	\\
	For any $\psi\in B$, we estimate
	\begin{equation*}
		\begin{aligned}
			&\left\|\Phi(\psi)(t)-\tilde{\psi}_0\right\|_{s} \\ 
			\leq & \left\|S(t) \tilde{\psi}_0-\tilde{\psi}_0\right\|_{s}+\left\|\int_0^t S(t-\tau)\left[-(1+i \mu)|\psi(\tau)|^{2 \sigma} \psi(\tau) +(r_1+ir_2)\langle \tilde{u}(\tau),Q\rangle\psi(\tau)\right] \d \tau\right\|_{s} \\
			\leq & M \|\tilde{\psi}_0 \|_{s}+(M+1)\int_0^t( r_1^2+r_2^2)^{1/2}\|\langle \tilde{u}(\tau), Q\rangle \psi(\tau)\|_{s}+(1+\mu^2)^{1/2}\left\||\psi(\tau)|^{2\sigma}\psi(\tau)\right\|_{s}\d\tau \\
			\leq & M \|\tilde{\psi}_0 \|_{s}+(M+1) ( r_1^2+r_2^2)^{1/2}C(Q) \int_0^t \sum_{i=1}^q\left|\tilde{u}_i(\tau)\right|\|\psi(\tau)\|_{s} \d \tau\\&+(M+1)(1+\mu^2)^{1/2}{C\left(s,d\right)}^{2\sigma}\int_0^t\|\psi(\tau)\|_{s}^{2\sigma+1} \d \tau,
		\end{aligned}
	\end{equation*}
	where we use \eqref{eq_BG} to obtain the third term in the last inequality. The Cauchy–Schwarz inequality, triangular inequality and the strict convexity of $f(x)=x^{2\sigma+1}, \forall x\geq0$, imply that 
	\begin{equation*}
		\begin{aligned}
			&\left\|\Phi(\psi)(t)-\tilde{\psi}_0\right\|_{s} \\ 
			\leq & M \|\tilde{\psi}_0 \|_{s}+(M+1) \sqrt{q} ( r_1^2+r_2^2)^{1/2}C(Q)\left(\int_0^t \sum_{i=1}^q\left|\tilde{u}_i(\tau)\right|^2 \d \tau\right)^{1 / 2}
			\left(2 \int_0^t \|\psi(\tau)-\tilde{\psi}_0 \|_{s}^2+ \|\tilde{\psi}_0 \|_{s}^2d \tau\right)^{1 / 2} \\
			& +(M+1)(1+\mu^2)^{1/2}{C\left(s,d\right) }^{2\sigma}2^{2\sigma} \int_0^t \|\psi(\tau)-\tilde{\psi}_0 \|_{s}^{2\sigma+1}+ \|\tilde{\psi}_0 \|_{s}^{2\sigma+1} \d \tau\\
			\leq& \frac{r(\tilde{\psi}_0)}{2} + (M+1) \sqrt{2q( r_1^2+r_2^2)}  C(Q) \|\tilde{u}\|_{L^2(0,1)}\left( \sup _{t \in [0,t_1]} \|\psi(t)-\tilde{\psi}_0 \|_{s} + \|\tilde{\psi}_0\|_s \right)\sqrt{t_1}\\
			&+(M+1)(1+\mu^2)^{1/2}{C\left(s,d\right) }^{2\sigma}2^{2\sigma}\left( \sup _{t \in [0,t_1]} \|\psi(t)-\tilde{\psi}_0 \|_{s}^{2\sigma+1} + \|\tilde{\psi}_0\|_s^{2\sigma+1} \right)t_1.
		\end{aligned}
	\end{equation*}
	Since $t_1\leq 1$, we obtain that $t_1\leq \sqrt{t_1}$ and 
	\begin{equation*}
		\begin{aligned}
			\left\|\Phi(\psi)(t)-\tilde{\psi}_0\right\|_{s} 
			\leq& \frac{r(\tilde{\psi}_0)}{2} + (M+1) \sqrt{2q( r_1^2+r_2^2)}  C(Q) \|\tilde{u}\|_{L^2(0,1)}\left( \sup _{t \in [0,t_1]} \|\psi(t)-\tilde{\psi}_0 \|_{s} + \|\tilde{\psi}_0\|_s \right)\sqrt{t_1}\\
			&+(M+1)(1+\mu^2)^{1/2}\left(C\left(s,d\right) 2\right)^{2\sigma}\left( \sup _{t \in [0,t_1]} \|\psi(t)-\tilde{\psi}_0 \|_{s}^{2\sigma+1} + \|\tilde{\psi}_0\|_s^{2\sigma+1} \right)\sqrt{t_1}\\
			\leq& \frac{r(\tilde{\psi}_0)}{2} + C(q,r_1,r_2, \tilde{u}, Q)\left( r(\tilde{\psi}_0) + \|\tilde{\psi}_0\|_s \right)\sqrt{t_1}\\
			&+(M+1)(1+\mu^2)^{1/2}\left(C\left(s,d\right) 2\right)^{2\sigma}\left(r(\tilde{\psi}_0)^{2\sigma+1} + \|\tilde{\psi}_0\|_s^{2\sigma+1} \right)\sqrt{t_1}\\
			\leq& \frac{r(\tilde{\psi}_0)}{2} + \frac{r(\tilde{\psi}_0)}{2}=r(\tilde{\psi}_0).
		\end{aligned}
	\end{equation*}
	Thus, we deduce that $ \Phi(\psi) \in B$.
	\\
	\textbf{Substep~2.2. $\Phi^n$ is a contraction over $B$ for $n$ large enough.}
	\\
	For any $\psi,\, \phi \in B$ and each $t\in (0,t_1)$,
	\begin{equation}\label{induction1}
		\begin{aligned}
			&\|\Phi(\psi)(t)-\Phi(\phi)(t)\|_{s}\\
			= & \left\|\int_0^t S(t-\tau)\left[-(1+i \mu)\left(|\psi(\tau)|^{2 \sigma} \psi(\tau)- |\phi(\tau)|^{2 \sigma} \phi(\tau)\right) +(r_1+ir_2)\langle \tilde{u}(\tau),Q\rangle\left(\psi(\tau)-\phi(\tau)\right)\right] \d \tau\right\|_{s} \\
			\leq & (M+1)(r_1^2+r_2^2)^{1/2}\int_0^t\|\langle \tilde{u}(\tau), Q\rangle(\psi(\tau)-\phi(\tau))\|_{s} \d \tau +(M+1)(1+\mu^2)^{1/2}C\left(s,d\right) \\
			&\times \int_0^t\|\psi(\tau)-\phi(\tau)\|_{s} \sum_{j=0}^{2\sigma}\|\psi(\tau)\|_{s}^j\|\phi(\tau)\|_{s}^{2\sigma-j} \d \tau \\
			\leq &(M+1) (r_1^2+r_2^2)^{1/2}C(Q) \sqrt{ q}\|\tilde{u}\|_{L^2(0,1)}\left(\int_0^t\|\psi(\tau)-\phi(\tau)\|_{s}^2 \d \tau\right)^{1 / 2} +(M+1)(1+\mu^2)^{1/2}C\left(s,d\right) 2^{2\sigma-2}\\
			& \times \int_0^t\|\psi(\tau)-\phi(\tau)\|_{s} D(\psi(\tau), \phi(\tau))\d \tau,
		\end{aligned}
	\end{equation}
	where
	\begin{equation*}
		D(\psi(\tau), \phi(\tau)):=\sum_{j=0}^{2\sigma}\left( \|\psi(\tau)-\tilde{\psi}_0 \|_{s}^j+ \|\tilde{\psi}_0 \|_{s}^j\right)\left( \|\phi(\tau)-\tilde{\psi}_0 \|_{s}^{2\sigma-j}+ \|\tilde{\psi}_0 \|_{s}^{2\sigma-j}\right).
	\end{equation*}
	Therefore, for every $0\leq t<\tau\leq t_1$, we get that 
	\begin{equation}\label{induction2}
		\begin{aligned}
			&\sup _{0 \leq t \leq \tau}\|\Phi(\psi)(t)-\Phi(\phi)(t)\|_{s}\\ \leq &(M+1) (r_1^2+r_2^2)^{1/2}C(Q) \sqrt{ q}\|\tilde{u}\|_{L^2(0,1)} \sqrt{\tau} \sup _{0 \leq t \leq \tau}\|\psi(t)-\phi(t)\|_{s} \\
			& +(M+1)(1+\mu^2)^{1/2}C\left(s,d\right) 2^{2\sigma-2} \sqrt{\tau} \sup _{0 \leq t \leq t_1} D(\psi(t), \phi(t)) \sup _{0 \leq t \leq \tau}\|\psi(t)-\phi(t)\|_{s} \\
			\leq & \left(C(q,r_1,r_2, \tilde{u}, Q)+(M+1)(1+\mu^2)^{1/2}C\left(s,d\right) 2^{2\sigma-2}\tilde{D} (\tilde{\psi}_0 )\right) \sqrt{\tau} \sup _{0 \leq t \leq \tau}\|\psi(t)-\phi(t)\|_{s},
		\end{aligned}
	\end{equation}
	where
	\begin{equation*}
		\tilde{D} (\tilde{\psi}_0 ):=\sum_{j=0}^{2\sigma}\left(r (\tilde{\psi}_0 )^j+ \|\tilde{\psi}_0 \|_{s}^j\right)\left(r (\tilde{\psi}_0 )^{2\sigma-j}+ \|\tilde{\psi}_0 \|_{s}^{2\sigma-j}\right).
	\end{equation*}
	Using \eqref{induction1} and \eqref{induction2}, and by induction on $ n $, we obtain
	\begin{equation*}
		\begin{aligned}
			\sup _{0 \leq t \leq t_1}\left\|\Phi^n(\psi)(t)-\Phi^n(\phi)(t)\right\|_{s} \leq&\left(C(q,r_1,r_2, \tilde{u}, Q)+(M+1)(1+\mu^2)^{1/2}C\left(s,d\right) 2^{2\sigma-2}\tilde{D} (\tilde{\psi}_0 )\right)^n \frac{\left(\sqrt{t_1}\right)^n}{\sqrt{n!}} \\ 
			&\times\sup _{0 \leq t \leq t_1}\|\psi(t)-\phi(t)\|_{s}.
		\end{aligned}
	\end{equation*}
	For $ n $ large enough, it holds that 
	\begin{equation*}
		\left(C(q,r_1,r_2, \tilde{u}, Q)+(M+1)(1+\mu^2)^{1/2}C\left(s,d\right) 2^{2\sigma-2}\tilde{D} (\tilde{\psi}_0 )\right)^n \frac{\left(\sqrt{t_1}\right)^n}{\sqrt{n!}}<1.
	\end{equation*}
	\\
	\textbf{Substep~2.3. The existence and extension of a solution.}
	\\
	By the contraction principle (see \cite[Theorem~5.7]{Brezis}), we deduce that $ \Phi $ has a unique fixed point $ \tilde{\psi}\in B $, which is the mild solution of problem~\eqref{(0.1)}, \eqref{(0.2)} in the form \eqref{(5.2)}.
	Furthermore, it holds that
	\begin{equation*}
		\sup _{t \in [0, t_1 ]}\|\tilde{\psi}(t)\|_{s} \leq(2 M+1) \|\tilde{\psi}_0 \|_{s}.
	\end{equation*}
	Therefore, we can conclude that, if $\tilde{\psi}$ is a mild solution of problem~\eqref{(0.1)}, \eqref{(0.2)} on the interval $[0, \tau]$, it can be extended to the interval $[0, \tau+\delta(\tau)]$ with $ \delta(\tau)>0 $. 
	In fact, by defining the quantities
	\begin{equation*}
		\begin{aligned}
			&M(\tau):=\sup \left\{\left\|S(t)-I\right\|_{\mathcal{L}\left(H^s; H^s\right)}\mid \tau \leq t \leq \tau+1\right\}, \\ &r (\tilde{\psi}(\tau)):=2 M(\tau) \|\tilde{\psi}(\tau) \|_{s} \\
			&C(\tau,q,r_1,r_2, \tilde{u}, Q):=(M(\tau)+1) \sqrt{2q( r_1^2+r_2^2)}  C(Q) \|\tilde{u}\|_{L^2(\tau,\tau+1)},\\
			&\mathcal{T}_3:= C(\tau,q,r_1,r_2, \tilde{u}, Q)\left(r (\tilde{\psi}(\tau) )+ \|\tilde{\psi}(\tau) \|_{s}\right),\\
			&	\mathcal{T}_4:=(M(\tau)+1)(1+\mu^2)^{1/2}\left(C\left(s,d\right) 2\right)^{2\sigma}\left(r(\tilde{\psi}(\tau))^{2\sigma+1} + \|\tilde{\psi}(\tau)\|_s^{2\sigma+1} \right),
		\end{aligned}
	\end{equation*}
	and 
	\begin{equation*}
		\mathcal{T}(\tau):=\min\left\{1,\quad\frac14\left(\frac{r (\tilde{\psi}(\tau) )}{\mathcal{T}_3+\mathcal{T}_4}\right)^2\right\},
	\end{equation*}
	we can define on $[\tau, \tau+\mathcal{T}(\tau)], \tilde{\psi}(t)=w(t)$ where $w(t)$ is the solution of the integral equation
	\begin{equation*}
		w(t)=S(t-\tau) \tilde{\psi}(\tau)+\int_\tau^t S(t-s) \left[-(1+i \mu)|w(s)|^{2 \sigma} w(s) +(r_1+ir_2)\langle \tilde{u}(s),Q\rangle w(s)\right] \d s,\,  \tau \leq t \leq \tau+\mathcal{T}(\tau).
	\end{equation*}
	
	Let $ [0,\tilde{T}) $ be the maximal existence interval of the mild solution $ \tilde{\psi} $ of problem~\eqref{(0.1)}, \eqref{(0.2)}, where $ \tilde{T}=\tilde{T}\left(\tilde{\psi}_0,\tilde{u}\right)>0 $. If $ \tilde{T}<+\infty $, then $ \|\tilde{\psi}(t)\|_s\rightarrow+\infty $ as
	$ t\rightarrow \tilde{T}^- $, otherwise $ \tilde{\psi} $ could be extended, which contradicts the maximality of $ \tilde{T} $. Moreover
	for any $0<T<\tilde{T}\left(\tilde{\psi}_0,\tilde{u}\right)$, we have
	\begin{equation}\label{bounded}
		\sup _{t \in[0, T]}\|\tilde{\psi}(t)\|_{s} \leq C(T) \|\tilde{\psi}_0 \|_{s}.
	\end{equation}
	\\
	\textbf{Step~3. Proof of the stability \eqref{stability} and the continuity \eqref{continuity}. }
	\\
	We first show \eqref{continuity}. Let $\psi,\, \tilde{\psi} \in C\left([0, T], H^s\right)$, with $0 \leq T < \min \left\{\tilde{T}\left(\psi_0,u\right),\, \tilde{T}\left(\tilde{\psi}_0,\tilde{u}\right)\right\}$, be the solutions of problem~\eqref{(0.1)}, \eqref{(0.2)} corresponding to the initial conditions $\psi_0$ and $\tilde{\psi}_0$ and controls $u$ and $\tilde{u}$, respectively. Then,
	\begin{equation*}
		\begin{aligned}
			\|\psi(t)-\tilde{\psi}(t)\|_{s} \leq & e^{Vt} \| \psi_0-\tilde{\psi}_0 \|_{s} + (1+ \mu^2)^{1/2}\int_0^t e^{V(t-\tau)}\left\||\psi(\tau)|^{2 \sigma} \psi(\tau)-|\tilde{\psi}(\tau)|^{2 \sigma} \tilde{\psi}(\tau) \right\|_s\d\tau \\
			&+(r_1^2+r_2^2)^{1/2}\int_0^te^{V(t-\tau)}\left\|\langle u(\tau),Q\rangle\psi(\tau)-\langle \tilde{u}(\tau),Q\rangle\tilde{\psi}(\tau)\right\|_s\d\tau\\
			\leq & e^{Vt} \| \psi_0-\tilde{\psi}_0 \|_{s} + (1+ \mu^2)^{1/2} e^{Vt}\int_0^t\left\||\psi(\tau)|^{2 \sigma} \psi(\tau)-|\tilde{\psi}(\tau)|^{2 \sigma} \tilde{\psi}(\tau) \right\|_s\d\tau \\
			&+(r_1^2+r_2^2)^{1/2}e^{Vt}\int_0^t\left\|\langle u(\tau),Q\rangle(\psi(\tau)-\tilde{\psi}(\tau)) \right\|_s\d\tau\\
			&+(r_1^2+r_2^2)^{1/2}e^{Vt}\int_0^t\left\|\langle u(\tau)-\tilde{u}(\tau),Q\rangle\tilde{\psi}(\tau)\right\|_s\d\tau.
		\end{aligned}
	\end{equation*}
	Using the inequality $\left\||\psi |^{2 \sigma} \psi -|\tilde{\psi} |^{2 \sigma} \tilde{\psi}  \right\|_s\leq C(s,d) \|\psi -\tilde{\psi} \|_s \sum_{j=0}^{2\sigma}\|\psi \|_s^j\|\tilde{\psi} \|_s^{2\sigma-j},\, \forall \psi,\tilde{\psi}\in H^s$, we obtain that
	\begin{equation*}
		\begin{aligned}
			&\|\psi(t)-\tilde{\psi}(t)\|_{s}  \\
			\leq & e^{Vt} \| \psi_0-\tilde{\psi}_0 \|_{s} + (1+ \mu^2)^{1/2} e^{Vt}C\left(s,d\right)\int_0^t\|\psi(\tau)-\tilde{\psi}(\tau)\|_s \sum_{j=0}^{2\sigma}\|\psi(\tau)\|_s^j\|\tilde{\psi}(\tau)\|_s^{2\sigma-j}\d\tau \\
			&+(r_1^2+r_2^2)^{1/2}e^{Vt}\int_0^t\left\|\langle u(\tau),Q\rangle(\psi(\tau)-\tilde{\psi}(\tau)) \right\|_s\d\tau\\
			&+(r_1^2+r_2^2)^{1/2}e^{Vt}\int_0^t\left\|\langle u(\tau)-\tilde{u}(\tau),Q\rangle\tilde{\psi}(\tau)\right\|_s\d\tau\\
			\leq & e^{Vt} \| \psi_0-\tilde{\psi}_0 \|_{s} + (1+ \mu^2)^{1/2} e^{Vt}C\left(s,d\right)\sqrt{t}\sup _{0 \leq \tau \leq t}\|\psi(\tau)-\tilde{\psi}(\tau)\|_{s}\\
			&\;\;\times\left(\sum_{j=0}^{2\sigma} \sup _{0 \leq \tau \leq t}\|\psi(\tau)\|_{s}^j \sup _{0 \leq \tau \leq t}\|\tilde{\psi}(\tau)\|_{s}^{2\sigma-j}\right) \\
			&+(r_1^2+r_2^2)^{1/2}e^{Vt}C(Q)\sqrt{2q}\|u\|_{L^2(0,t)}\sqrt{t}\sup _{0 \leq \tau \leq t}\|\psi(\tau)-\tilde{\psi}(\tau)\|_{s}\\
			&+(r_1^2+r_2^2)^{1/2}e^{Vt}C(Q)\sqrt{2q}\|u-\tilde{u}\|_{L^2(0,t)}\sqrt{t}\sup _{0 \leq \tau \leq t}\| \tilde{\psi}(\tau)\|_{s} .
		\end{aligned}
	\end{equation*}
	Hence for any $t_2\in(0,T)$, we have
	\begin{equation*}
		\begin{aligned}
			&\sup _{0 \leq t \leq t_2}\|\psi(t)-\tilde{\psi}(t)\|_{s} \\ \leq & e^{VT}  \|\psi_0-\tilde{\psi}_0 \|_{s}+\left((r_1^2+r_2^2)^{1/2}e^{VT}C(Q)\sqrt{2q}\sqrt{T}\sup _{0 \leq \tau \leq T}\| \tilde{\psi}(\tau)\|_{s}\right)\|u-\tilde{u}\|_{L^2(0,T)}  \\
			& + (1+ \mu^2)^{1/2} e^{VT}C\left(s,d\right) \left(\sum_{j=0}^{2\sigma} \sup _{0 \leq \tau \leq T}\|\psi(\tau)\|_{s}^j \sup _{0 \leq \tau \leq T}\|\tilde{\psi}(\tau)\|_{s}^{2\sigma-j}\right)\sqrt{t_2} \sup _{0 \leq t \leq t_2}\|\psi(t)-\tilde{\psi}(t)\|_{s} \\
			&+ (r_1^2+r_2^2)^{1/2}e^{VT}C(Q)\sqrt{2q}\|u\|_{L^2(0,T)}\sqrt{t_2}\sup _{0 \leq t \leq t_2}\|\psi(t)-\tilde{\psi}(t)\|_{s},
		\end{aligned}
	\end{equation*}which, along with \eqref{bounded}, implies that
	\begin{equation*}
		\begin{aligned}
			&\sup _{0 \leq t \leq t_2}\|\psi(t)-\tilde{\psi}(t)\|_{s} \\ 
			\leq & e^{VT}  \|\psi_0-\tilde{\psi}_0 \|_{s}+\left((r_1^2+r_2^2)^{1/2}C(Q)\sqrt{2q}\Lambda\right)e^{VT}\sqrt{T}\|u-\tilde{u}\|_{L^2(0,T)}  \\
			& + (1+ \mu^2)^{1/2} C\left(s,d\right) \left(\sum_{j=0}^{2\sigma} C(T)^j\left(\delta+\|\tilde{\psi}_0\|_s\right)^j \Lambda^{2\sigma-j}\right)e^{VT} \sqrt{t_2} \sup _{0 \leq t \leq t_2}\|\psi(t)-\tilde{\psi}(t)\|_{s} \\
			&+ (r_1^2+r_2^2)^{1/2}C(Q)\sqrt{2q}\left(\delta+\Lambda\right)e^{VT}\sqrt{t_2}\sup _{0 \leq t \leq t_2}\|\psi(t)-\tilde{\psi}(t)\|_{s}.
		\end{aligned}
	\end{equation*} Choosing $ t_2 $ to be such that
	\begin{equation*}
		(1+ \mu^2)^{1/2} C\left(s,d\right) \left(\sum_{j=0}^{2\sigma} C(T)^j\left(\delta+\|\tilde{\psi}_0\|_s\right)^j \Lambda^{2\sigma-j}\right)e^{VT} \sqrt{t_2}+ (r_1^2+r_2^2)^{1/2}C(Q)\sqrt{2q}\left(\delta+\Lambda\right)e^{VT}\sqrt{t_2}=\frac12,
	\end{equation*}
	we obtain
	\begin{equation}\label{eq_continuity_in_(0,t_1)}
		\sup _{0 \leq t \leq t_2}\|\psi(t)-\tilde{\psi}(t)\|_{s} \leq C(T,\Lambda) \left(\|\psi_0-\tilde{\psi}_0 \|_{s}+ \|u-\tilde{u}\|_{L^2(0,T)}\right),
	\end{equation}where
	\begin{equation*}
		C(T,\Lambda):=2  e^{VT} + 2\left((r_1^2+r_2^2)^{1/2}C(Q)\sqrt{2q}\Lambda\right)e^{VT}\sqrt{T}.
	\end{equation*}
	Similarly, we can always obtain another, maybe much bigger positive constant, and still denote it by $C(T,\Lambda)>0$, such that
	\begin{equation*}
		\sup _{jt_2 \leq t \leq jt_2+t_2}\|\psi(t)-\tilde{\psi}(t)\|_{s} \leq  C(T,\Lambda)\left(\|\psi_0-\tilde{\psi}_0 \|_{s}+\|u-\tilde{u}\|_{L^2(0,T)}\right), \quad \forall j=1,...,[T/t_2]-1,
	\end{equation*}
	\begin{equation*}
		\sup _{[T/t_2]t_2 \leq t \leq T}\|\psi(t)-\tilde{\psi}(t)\|_{s} \leq  C(T,\Lambda)\left(\|\psi_0-\tilde{\psi}_0 \|_{s}+\|u-\tilde{u}\|_{L^2(0,T)}\right),
	\end{equation*}which, along with \eqref{eq_continuity_in_(0,t_1)}, implies \eqref{continuity}.
	Finally, the same technique also implies the validity of \eqref{stability}.
\end{proof}

\end{document}